\documentclass[10pt]{amsart}
\usepackage{enumerate,amsmath,amssymb,latexsym,
amsfonts, amsthm, amscd, MnSymbol}

\theoremstyle{plain}

\newtheorem{theorem}{Theorem}

\numberwithin{equation}{section}

\newcommand{\LL}{\mathbb{L}}
\newcommand{\da}{\vcrossing} 
\newcommand{\mtd}{\medtriangledown}

\addtocounter{section}{-1}

\begin{document}

\title {Projective space: reguli and projectivity}

\date{}

\author[P.L. Robinson]{P.L. Robinson}

\address{Department of Mathematics \\ University of Florida \\ Gainesville FL 32611  USA }

\email[]{paulr@ufl.edu}

 \keywords{}

\begin{abstract}

We investigate an `assumption of projectivity' that is appropriate to the self-dual axiomatic formulation of three-dimensional projective space. 

\end{abstract}

\maketitle

\medbreak

\section{Introduction} 

\medbreak 

In the traditional Veblen-Young [VY] formulation of projective geometry, based on assumptions of alignment and extension, it is proved that a projectivity fixing three points on a line fixes all harmonically related points on the line. The Veblen-Young assumptions of alignment and extension are then augmented by a (provisional) `assumption of projectivity' according to which a projectivity that fixes three points on a line fixes all points on the line. See [VY] Chapter IV and Section 35 in particular. 

\medbreak 

In [R] we proposed a self-dual formulation of three-dimensional projective space, founded on lines and abstract incidence, with points and planes as derived notions. The version of the `assumption of projectivity' appropriate to this formulation of projective space should itself be self-dual. Such a version is ready to hand in [VY]: indeed, Veblen and Young note that (with alignment and extension assumed) their `assumption of projectivity' is equivalent to Chapter XI Theorem 1; and this theorem is expressed purely in terms of lines and incidence. 

\medbreak 

Our purpose in this brief note is to indicate some of the results surrounding this self-dual `assumption of projectivity', presenting them within the axiomatic framework of [R]. 

\medbreak 

\section{Framework} 

\medbreak 

We recall briefly the axiomatic framework of [R]. The set $\LL$ of {\it lines} is provided with a symmetric reflexive relation $\da$ of {\it incidence} satisfying AXIOM [1] - AXIOM [4] below. For convenience, when $S \subseteq \LL$ we write $S^{\da}$ for the set comprising all lines that are incident to each line in $S$; in case $S = \{ l_1, \dots , l_n \}$ we write $S^{\da} = [ l_1 \dots  l_n ]$. Further, when the lines $a, b \in \LL$ are not incident we call them {\it skew} and write $a \; | \; b.$ 

\medbreak 

$\bullet$ AXIOM [1]: For each line $l$ the set $l^{\da}$ contains three pairwise skew lines. 

\medbreak 

$\bullet$ AXIOM [2]: For each incident pair of distinct lines $a, b:$ \par
\indent \indent \indent \indent \indent [2.1]  the set $[a b]$ contains skew pairs of lines; \par 
\indent \indent \indent \indent \indent [2.2]  if $c \in [a b] \setminus [a b]^{\da}$ is one of such a skew pair then $[a b c]$ contains no skew pairs; \par 
\indent \indent \indent \indent \indent [2.3] if $x, y$ is a skew pair in $[ab]$ then $[a b] = [a b x] \cup [a b y].$

\medbreak 

$\bullet$ AXIOM [3]: If $a, b$ is an incident line pair and $c \in [a b] \setminus [a b]^{\da}$ then there exist an incident line pair $p, q$ and $r \in [p q] \setminus [p q]^{\da}$ such that $[a b c] \cap [p q r] = \emptyset.$

\medbreak 

$\bullet$ AXIOM [4]: Whenever $a, b$ and $p, q$ are pairs of distinct incident lines, 
$$(a \upY b) \cap (p \upY q) \neq \emptyset, \; \; \; 
(a \mtd b) \cap (p \mtd q) \neq \emptyset.$$

\medbreak 

Regarding this last axiom, we remark that on the set $\Sigma (a, b) = [a b] \setminus [a b]^{\da}$ (comprising all lines that are one of a skew pair in $[a b]$) incidence restricts to an equivalence relation having two equivalence classes, which we denote by $\Sigma_{\upY} (a, b)$ and $\Sigma_{\mtd} (a, b)$: the {\it point} $a \upY b = [a b c_{\upY} ]$ does not depend on the choice of $c_{\upY} \in \Sigma_{\upY} (a, b)$; likewise, the {\it plane} $a \mtd b = [a b c_{\mtd}]$ is independent of $c_{\mtd} \in \Sigma_{\mtd} (a, b)$. 

\medbreak 

Among its virtues, this self-dual axiomatization of projective space places points and planes on a manifestly equal footing and makes the principle of duality particularly transparent. 

\medbreak 

We adopt the usual notation and extended terminology regarding incidence as it relates to lines, points and planes. For instance, we say that the point $P$ and the plane $\pi$ are incident precisely when the intersection $P \cap \pi$ is nonempty (in which case this intersection contains more than one line); we may instead say that $P$ lies on $\pi$ or that $\pi$ passes through $P$. For another instance, points and/or planes are collinear precisely when they are on (that is, contain) one and the same line. 

\medbreak 

\section{Lemmata} 

\medbreak 

We assemble here some results that will be useful in what follows. The results themselves are standard and indeed may be found in [VY] Chapter 1: the first is the Corollary to Theorem 6 on page 20, the second is Theorem 9 on page 22 and the third is Exercise 1 on page 25;  but we offer proofs within the self-dual framework. Notice that Theorem \ref{lP} and Theorem \ref{lpi} are actually dual, so a proof of the one yields a proof of the other by duality; nevertheless, we elect to offer separate proofs for the sake of variety. 

\medbreak 

\begin{theorem} \label{lP}
If the line $l$ does not pass through the point $P$ then there exists a unique plane $Pl$ through both $l$ and $P$. 
\end{theorem} 

\begin{proof} 
AXIOM [1] furnishes skew lines $x$ and $y$ in $l^{\da}$. The points $P$, $X = l \upY x$ and $Y = l \upY y$ are not collinear, for $X \neq Y$ so ([R] Theorem 14) the only line through $X$ and $Y$ is $l$ which does not pass through $P$. Let $P \cap X = \{ a \}$ and $P \cap Y = \{ b \}$ (by [R] Theorem 14 again). According to [R] Theorem 16, $l \in \Sigma_{\mtd} (a, b)$ and the plane $a \mtd b = [a b l]$ passes through $l$ and through $P$ (which contains $a, b \in a \mtd b$). So much for existence; now for uniqueness. Let $\pi$ be a plane through $P$ and $l$: then $\pi$ passes through $P$, $X (\ni l)$ and $Y (\ni l)$; so [R] Theorem 16 forces $\pi = a \mtd b$. 
\end{proof} 

\begin{theorem} \label{lpi}
If the line $l$ does not lie in the plane $\pi$ then there exists a unique point on both $l$ and $\pi$. 
\end{theorem} 

\begin{proof} 
Choose (by AXIOM [1]) any line $m$ incident to $l$. The planes $\pi$ and $l \mtd m$ are distinct, because $l$ lies in the latter but not in the former; by [R] Theorem 14 it follows that $\pi \cap (l \mtd m)$ is a singleton, say $\{ n \}$. Now the point $l \upY n$ plainly lies on $l$ and lies on $\pi$ because $(l \upY n) \cap \pi$ contains $n$. This proves existence. If $Q$ were a second point on $l$ and $\pi$ then $P \cap Q$ would be the singleton $\{ l \}$ by [R] Theorem 14 and then [R] Theorem 15 would place $l$ on $\pi$ contrary to hypothesis. This proves uniqueness. 
\end{proof} 

\begin{theorem} \label{Puv}
If $u$ and $v$ are skew and $P$ is a point not on either, then there exists a unique line through $P$ meeting $u$ and $v$. 
\end{theorem} 

\begin{proof} 
The planes $P u$ and $P v$ (Theorem \ref{lP}) are distinct: as $u$ and $v$ are skew, the equality $Pu = Pv$ would violate the `Pasch' property in AXIOM [2.2]. [R] Theorem 14 tells us that the  intersection $Pu \cap Pv$ contains a unique line $m$. AXIOM [2.2] ensures that $m \: (\in Pu \cap Pv)$ is incident to $u \: (\in Pu$) and to $v \: (\in Pv$). The dual of Theorem 15 in [R] ensures that $P$ lies on the line $m$ common to $Pu$ and $Pv$. To see uniqueness, let also $n \in P \cap [u v]$; of course, $n \notin \{ u, v \}$. The plane $P u$ contains the distinct points $P \: (\ni v)$ and $u \upY n \: (\nowns v)$ on $n$ and hence contains the line $n$ itself by [R] Theorem 15; $P v$ contains $n$ likewise, so $n \in Pu \cap Pv = \{ m \}$. 
\end{proof} 

\medbreak 

\section{Reguli} 

\medbreak 

In this section, we introduce the notion in terms of which our `assumption of projectivity' is expressed. 

\medbreak 

Let $u, v, w \in \LL$ be pairwise skew lines. We call the set $[u v w] : = \{ u, v, w \}^{\da}$ a {\it regulus} of which the lines $u, v, w$ are {\it directrices}. 

\medbreak 

\begin{theorem} \label{three}
If the lines $u, v, w \in \LL$ are pairwise skew, then $[u v w]$ is nonempty. 
\end{theorem} 

\begin{proof} 
Pick a line $p$ incident to $w$ by Axiom [1] and let $P = p \upY w$. As $P$ lies on neither $u$ nor $v$, Theorem \ref{Puv} provides us with a line $m$ through $P$ (hence incident to $w$) meeting $u$ and $v$. Thus $m$ lies in $[u v w]$ and so $[u v w]$ is nonempty. 
\end{proof} 

Remark: In fact, AXIOM [1] puts three skew lines $p_1, p_2, p_3$ in $w^{\da}$ and so places three points $P_1, P_2, P_3$ on $w$; the proof of Theorem \ref{three} then yields three lines $m_1, m_2, m_3$ in $[u v w]$. Further, recall from [R] Theorem 7 that distinct lines in a regulus are necessarily skew.

\medbreak 

Note that reguli do exist; indeed, each line is contained in some regulus. Let $l \in \LL$ be any line: AXIOM [1] provides three pairwise skew lines $u, v, w$ incident to $l$; symmetry of incidence then places $l$ in the regulus $[u v w]$. We can go further than this: each skew pair of lines is contained in some regulus. 

\begin{theorem} \label{two} 
If the lines $u, v$ are skew then there exist pairwise skew lines $l_1, l_2, l_3$ such that $u, v \in [l_1 l_2 l_3]$. 
\end{theorem} 

\begin{proof} 
AXIOM [1] gives pairwise skew triples: $u_1, u_2, u_3$ incident to $u$ and $v_1, v_2, v_3$ incident to $v$. The points $u \upY u_3 \: (\ni u)$ and $v \upY v_3 \: (\nowns u)$ are distinct, so $(u \upY u_3) \cap (v \upY v_3)$ is a singleton by [R] Theorem 14; say $(u \upY u_3) \cap (v \upY v_3) = \{ l_3 \}$ with $(u \upY u_1) \cap (v \upY v_1) = \{ l_1 \}$ and $(u \upY u_2) \cap (v \upY v_2) = \{ l_2 \}$ likewise. Evidently, $u$ and $v$ lie in $[l_1 l_2 l_3]$ so we need only verify that $l_1, l_2, l_3$ are pairwise skew. First, they are distinct: if $l_1 = l_2$ then $u \upY u_1 \ni l_1 = l_2 \in u \upY u_2$ which with $(u \upY u_1) \cap (u \upY u_2) = \{ u \}$ would force $l_1 = l_2 = u$; but $l_1$ and $l_2$ meet $v$ whereas $u$ does not. Now, assume that $l_1$ meets $l_2$ (and aim at another contradiction). Certainly, $u, v$ is a skew pair in $[l_1 l_2]$. By [R] Theorem 4, the hypothesis $u \in l_1 \upY l_2$ would imply $u_1 \in u \upY l_1 = l_1 \upY l_2 = l_2 \upY u \ni u_2$  which contradicts AXIOM [2.2] together with the fact that $u_1$ and $u_2$ are skew; thus $u \in [l_1 l_2] \setminus (l_1 \upY l_2)$  and so $u \in l_1 \mtd l_2$ by AXIOM [2.3] while $v \in l_1 \mtd l_2$ in similar fashion. As AXIOM [2.2] prevents $l_1 \mtd l_2$ from containing a skew pair, we have a final contradiction. 
\end{proof} 

As a companion result, note that if $u, v$ is a skew pair in the regulus $[l_1 l_2 l_3]$ then there exists $w \in [l_1 l_2 l_3] \setminus \{ u, v \}$: indeed, $[l_1 l_2 l_3]$ contains at least three lines as noted after Theorem \ref{three}; at least one of these must be different from $u$ and $v$. Similarly, if $w \in [l_1 l_2 l_3]$ then there exists a (necessarily) skew pair $x, y \in [l_1 l_2 l_3] \setminus \{ w \}$. 

\medbreak 

Up to this point, there is no assurance that the regulus obtained by choosing as directrices three lines in the regulus $[u v w]$ is independent of the three lines that are chosen; we address this in the next section.

\medbreak 

\section{Projectivity} 

\medbreak 

We now introduce an `assumption of projectivity' that is appropriate to our self-dual axiomatization of projective space. The following is a restatement of Theorem 1 in Chapter XI of [VY]. 

\medbreak 

$\bullet$ AXIOM [P1]: Let $u, v, w$ be pairwise skew lines and let $a, b, c, d$ be lines in the regulus $[u v w]$. If a line meets three of $a, b, c, d$ then it meets the fourth. 

\medbreak 

Notice that this refers neither to points nor to planes but only to lines and (abstract) incidence; it therefore fits perfectly into our self-dual axiomatization of projective space. Consider also the following statement, with the same self-dual character. 

\medbreak 

$\bullet$ AXIOM [P2]: Let $u, v, w$ be pairwise skew and let $x, y, z \in [u v w]$ be distinct. If $l \in [u v w]$ and $m \in [x y z]$ then $l$ meets $m$. 

\medbreak 

Here, the distinct lines $x, y, z$ are themselves pairwise skew, as noted after Theorem \ref{three}. Accordingly, both $[u v w]$ and $[x y z]$ are reguli. 

\begin{theorem} 
AXIOM [P1] and AXIOM [P2] are equivalent. 
\end{theorem} 

\begin{proof} 
$[P1] \Rightarrow [P2]$. Take the lines $x, y, z, l$ in the regulus $[u v w]$. Already, $m$ meets the three lines $x, y, z$; so $m$ meets the fourth line $l$. \par 
$[P2] \Rightarrow [P1]$. We may assume $a, b, c, d$ distinct. Say $l$ meets $a, b, c$: then $d \in [u v w]$ and $l \in [a b c]$ so $d$ meets $l$ as required. 
\end{proof} 

It is therefore legitimate to take either AXIOM [P1] or AXIOM [P2] as our `assumption of projectivity'. 

\medbreak 

We now derive some essentially standard consequences of this assumption, but working within the self-dual axiomatic framework. 

\medbreak 

\begin{theorem} \label{regeq} 
Let the lines $u, v, w$ be pairwise skew. If $x_1, y_1, z_1 \in [u v w]$ are distinct and $x_2, y_2, z_2 \in [u v w]$ are distinct, then $[x_1 y_1 z_1] = [x_2 y_2 z_2]$. 
\end{theorem} 

\begin{proof} 
Let $l \in [x_1 y_1 z_1]$: then $l$ meets the first three of the lines $x_1, y_1, z_1, z_2 \in [u v w]$ and so (by our `assumption of projectivity') meets the last $z_2$; in like manner, $l$ meets $x_2$ and $y_2$. Thus $[x_1 y_1 z_1] \subseteq [x_2 y_2 z_2]$ and symmetry concludes the argument. 
\end{proof} 

The regulus {\it conjugate} to $[u v w]$ is the regulus $[x y z]$ well-defined by any choice of triple $x, y, z \in [u v w]$. Observe that $[u v w]$ is then the regulus conjugate to $[x y z]$ simply because $u, v, w$ is a triple in $[x y z]$. 

\medbreak 

Recall from Theorem \ref{two} that any two skew lines are contained in some regulus. The following is a standard further companion to this result. 

\begin{theorem} 
Two distinct reguli can have at most two lines in common. 
\end{theorem} 

\begin{proof} 
Let the reguli $[u_1 v_1 w_1]$ and $[u_2 v_2 w_2]$ have the three lines $x, y, z$ in common; then 
$$u_1, v_1, w_1 \in [x y z] \ni u_2, v_2 w_2$$
and so $[u_1 v_1 w_1] = [u_2 v_2 w_2]$ by Theorem \ref{regeq}. 
\end{proof} 

The following property of reguli is also standard; our `assumption of projectivity' appears here in the very definition of the conjugate. 

\begin{theorem} 
If the point $P$ lies on a line of the regulus $[u v w]$ then $P$ lies on a line of its conjugate regulus. 
\end{theorem} 

\begin{proof} 
Let $P$ lie on $l \in [u v w]$. As noted after Theorem \ref{two}, we may choose two lines $m, n \in [u v w]$ distinct from $l$; note that $[l m n]$ is then the regulus conjugate to $[u v w]$. As $l$ and $m$ are skew, AXIOM [2.2] places $P$ off $m$ (that is, $m \notin P$) and then Theorem \ref{lP} passes a unique plane $Pm$ through $P$ and $m$. As $m$ and $n$ are skew, AXIOM [2.2] places $n$ off $Pm$ (that is, $n \notin Pm$) and then Theorem \ref{lpi} situates a unique point $Q$ on $n$ and $Pm$. The points $P$ and $Q$ are distinct, as the former lies on $l$ and the latter on $n$; consequently, [R] Theorem 14 runs a unique line $r$ through $P$ and $Q$, and $r$ lies in $Pm$ by [R] Theorem 15.  The line $r$ meets $l \: (\in P$) and $n \: (\in Q$); AXIOM [2.2] also forces $r \in Pm$ and $m \in Pm$ to meet. As $P$ lies on the line $r$ in the conjugate regulus $[l m n]$ we are done. 
\end{proof} 

Dually, any plane that contains a line of a regulus also contains a line of the conjugate regulus.

\begin{center} 
{\small R}{\footnotesize EFERENCES}
\end{center} 
\medbreak 

[R] P.L. Robinson, {\it Projective Space: Lines and Duality}, arXiv 1506.06051 (2015). 

\medbreak 

[VY] Oswald Veblen and John Wesley Young, {\it Projective Geometry}, Volume I, Ginn and Company, Boston (1910). 

\medbreak

\end{document}